\documentclass[10pt]{amsart}
\usepackage{amssymb,latexsym,amsmath,amsfonts,amsthm}
\usepackage[mathscr]{eucal}
\usepackage{mathrsfs}
\usepackage{amscd}
\usepackage{amsxtra}
\usepackage[nointegrals]{wasysym}
\usepackage{bm}
\usepackage{calc}
\usepackage{float}
\usepackage[usenames,dvipsnames]{xcolor}
\usepackage{hyperref}
\hypersetup{colorlinks=true, citecolor=MidnightBlue, linkcolor=BrickRed}
\usepackage{mathtools}

\makeatletter
\newtheorem*{rep@theorem}{\rep@title}
\newcommand{\newreptheorem}[2]{%
\newenvironment{rep#1}[1]{%
 \def\rep@title{#2 \ref{##1}}%
 \begin{rep@theorem}}%
 {\end{rep@theorem}}}
\makeatother

\numberwithin{equation}{section}
\theoremstyle{definition}

\theoremstyle{definition}

\newtheorem*{rmk}{Remark}
\theoremstyle{plain}
\newtheorem{theorem}{Theorem}[section]
\newreptheorem{theorem}{Theorem}
\newtheorem{lemma}[theorem]{Lemma}
\newreptheorem{lemma}{Lemma}
\newtheorem{cor}[theorem]{Corollary}
\newtheorem{Prop}[theorem]{Proposition}

\DeclarePairedDelimiter\floor{\lfloor}{\rfloor}

\newcommand{\beas}{\begin{eqnarray*}}
\newcommand{\eeas}{\end{eqnarray*}}
\newcommand{\bes} {\begin{equation*}}
\newcommand{\ees} {\end{equation*}}
\newcommand{\be} {\begin{equation}}
\newcommand{\ee} {\end{equation}}
\newcommand{\bea} {\begin{eqnarray}}
\newcommand{\eea} {\end{eqnarray}}

\newcommand{\bc}{\mathcal{M}}
\newcommand{\bt}{\boldsymbol{t}}

\newcommand{\D}{\mathbb{D}}
\newcommand{\clD}{\overline\D}
\newcommand{\eps}{\varepsilon}
\newcommand{\om}{\omega}

\newcommand{\V}{\mathcal{V}}
\newcommand{\wbar}{\overline{w}}

\newcommand{\zt}{\zeta}
\newcommand{\zbar}{\overline{z}}

\newcommand\partl[2]{\dfrac{\partial{#1}}{\partial{#2}}}

\newcommand\secpartl[3]{\dfrac{\partial^2{#1}}{\partial{#2}\partial{#3}}}

\newcommand{\Chess}{\operatorname{Hess}_{ \C}\!\:}
\newcommand{\JacC}{\operatorname{J}_\C\!\:}

\newcommand{\cont}{\mathcal{C}}
\newcommand{\hol}{\mathcal{O}}

\newcommand{\rea}{\operatorname{Re}}
\newcommand{\ima}{\operatorname{Im}} 

\newcommand{\wt}{\widetilde}
\newcommand{\dist}{\operatorname{dist}}

\newcommand{\CC}{\mathbb{C}^2}
\newcommand{\Cn}{\mathbb{C}^n}
\newcommand{\C} {\mathbb{C}} 

\newcommand{\rl}{\mathbb{R}}

\newcommand{\pnat} {\mathbb{N}_+} 
\newcommand{\N} {\mathbb{N}}

\newcommand{\std}{_{\operatorname{st}}}

\begin{document}

\title[Polynomially convex embeddings]{Polynomially convex embeddings of even-dimensional compact manifolds}

\author{Purvi Gupta \and Rasul Shafikov}
\address{Department of Mathematics, Rutgers University\\
New Brunswick, NJ 08854, U.S.A.\\
purvi.gupta@rutgers.edu }

\address{Department of Mathematics, University of Western Ontario\\
Middlesex College,  London, Ontario N6A 5B7, Canada\\
shafikov@uwo.ca}

\begin{abstract} The totally-real embeddability of any $2k$-dimensional compact manifold $M$ into $\C^{n}$, $n\geq 3k$, has several consequences: the genericity of polynomially convex embeddings of $M$ into $\Cn$, the existence of $n$ smooth generators for the Banach algebra $\cont(M)$, the existence of nonpolynomially convex embeddings with no analytic disks in their hulls, and the existence of special plurisubharmonic defining functions. We show that these results can be recovered even when $n=3k-1$, $k>1$, despite the presence of complex tangencies, thus lowering the known bound for the optimal $n$ in these (related but inequivalent) questions. 
\end{abstract}

\maketitle

\section{Introduction and main results}\label{sec_intro}

Polynomial convexity is an important notion largely owing to the Oka-Weil theorem which states that holomorphic functions in a neighbourhood 
of a polynomially convex set $M$ (see Section~\ref{sec_backmat} for relevant definitions) can be approximated uniformly on $M$ by holomorphic polynomials. Although polynomial convexity imposes topological restrictions on $M$, it is known that if $M$ is a 
nonmaximally totally real submanifold of $\C^n$,
it can be deformed via a small perturbation into a polynomially convex one, as proved by Forstneri{\v c}-Rosay~\cite{FoRo93}, 
Forstneri{\v c}~\cite{Fo94}, and L\o w-Wold~\cite{LoWo09}. The condition that {\it any} abstract 
$m$-dimensional compact real manifold admits a totally real embedding into $\mathbb C^n$ is well understood: one must have $\floor{\frac{3m}{2}} \le n$. Thus, any $m$-dimensional compact manifold can be embedded as a totally real
polynomially convex submanifold of $\mathbb C^n$ provided that $n\ge \floor{\frac{3m}{2}}$ and $(m,n)\neq (1,1)$. 
  
The bound discussed above is sharp for manifolds without boundary, see~\cite{HoJaLa12}. That is, if $n<\floor{\frac{3m}{2}}$, then certain $m$-dimensional compact manifolds necessarily acquire complex tangent directions when embedded into $\Cn$. The  
points where the tangent space of $M\subset\Cn$ contains complex directions are called the {\it CR-singularities} of $M$. 
CR-singularities encode topological information about $M$, such as its 
Euler characteristic and Pontryagin numbers; see Lai~\cite{La72}, Webster~\cite{We85}, and Domrin~\cite{Do95}. The simplest nontrivial case of CR-singularities is that of {\it complex points} of a real surface in $\mathbb C^2$, first studied in the seminal work of Bishop~\cite{Bi65}. Different types of complex points can endow the surface with different local convexity properties (see Section \ref{sec_backmat} for derails). Regardless of this, a surface in $\mathbb C^2$ can never be globally polynomially convex.

In this paper we consider the only other case when CR-singularities are generically discrete and $m<n$, namely when $m=2k$ and 
$n=3k-1$, $k>1$, (if $m\geq  n$, a smooth $M\subset\Cn$ can never be polynomially convex; see Stout~\cite[Section~2.3]{St07}). 
Beloshapka~\cite{Be97} for $k=2$, and Coffman~\cite{Co97} for all $k\geq 2$, constructed the normal form \eqref{eq_BCform} for generic 
CR-singularities of this kind. Our principal result is to show
that, unlike the case of complex points of real surfaces, $M$ is locally polynomially convex near {\em any} such CR-singularity, 
and as a result, there exists a polynomially convex embedding of $M$ in $\C^{3k-1}$. More precisely the following holds.

\begin{theorem}\label{thm_main} 
Suppose $M$ is a $2k$-dimensional ($k>1$) smooth compact connected submanifold (closed or with boundary) 
of $\C^{3k-1}$. Then, given any $s\geq 2$, there exists a $\cont^s$-small perturbation $M'$ of $M$ that is polynomially convex. The submanifold 
$M'$ is totally real with finitely many generic CR-singularities.
\end{theorem}

The question of the optimal $n$ that allows polynomially convex smooth embeddability of all $m$-dimensional manifolds into $\Cn$ was raised in \cite[Question 4.]{IzSt18}. Theorem \ref{thm_main} improves previously known bounds. We note that if the embedding is merely required to be topological, then Vodovoz and Zaidenberg have shown that the optimal value of $n$ is $m+1$ for all $m\geq 1$ (see \cite{VoZa71}). Our proof is based on the idea of perturbation of $M$ away from the set of CR- singularities where $M$ is already locally polynomially 
convex; a general result of this type is contained in Arosio-Wold~\cite{ArWo17}.
When $M$ has nonempty boundary, $M'$ can be further perturbed to be totally real and polynomially 
convex. This can be done by `pushing' any CR-singularity of $M'$ to one of its boundary components and then removing a thin collar neighbourhood of the boundary, leaving the manifold with no CR-singularities. A small perturbation can now be used to further make it polynomially convex. 

We now use Theorem~\ref{thm_main} to produce generators of the Banach algebra of continuous complex-valued functions over a smooth compact
real manifold. First consider an elementary example. Any continuous function on the circle $S^1\subset\C_z$ can be uniformly approximated on $S^1$ by a sequence of polynomial combinations of $z$ and $1/z$. This follows from the Stone-Weierstrass approximation theorem.  Generally, given a real manifold $M$ we say that $\cont^\ell(M)$, the space of $\ell$-times continuously differentiable functions on $M$, has {\em $n$-polynomial density} if there is a tuple $F=(f_1,...,f_n)$ of $n$ functions 
in $\cont^\infty(M)$ such that the set
\bes
\{P\circ F :P\text{ is a holomorphic polynomial on}\ \Cn\}
\ees
is dense in $\cont^\ell(M)$. If $F$ exists, we call $\{f_1,...,f_n\}$ a {\em PD-basis} of $\cont^\ell(M)$. The notions of rational density and an RD-basis can be defined analogously. The existence of $2$-RD bases for surfaces is discussed in Shafikov-Sukhov~\cite{ShSu15}. The combined use of the Oka-Weil theorem and an approximation result by Nirenberg-Wells~\cite[Theorem 1]{NiWe67} shows that the components of a totally real and polynomially convex embedding
$F: M \hookrightarrow \mathbb C^n$ give a PD-basis of $C^\ell(M)$ (see \cite{GuSh17} for more details). Thus, any compact real $m$-manifold has $n=\floor{\frac{3m}{2}}$-polynomial density. For $\ell\geq 1$, this is the optimal value of $n$ for which $\cont^\ell(M)$ has $n$-polynomial density for {\it all} $m$-dimensional manifolds, but for $\ell=0$ the optimal $n$ appears to be somewhere in the range $m<n<\floor{\frac{3m}{2}}$ (see \cite{VoZa71} for the case of continuous generators). While it is an open problem to find this optimal $n$, Theorem~\ref{thm_main} gives the following improvement for even-dimensional manifolds. 

\begin{cor}\label{cor_dens}
Let $M$ be a $2k$-dimensional ($k>1$) compact manifold. Then, $\cont(M)$ has $(3k-1)$-polynomial density. Further, if $M$ has nonempty boundary, then $\cont^\ell(M)$ has $(3k-1)$-polynomial density for all $\ell\geq 0$.
\end{cor}

Our techniques also allow us to improve another dimensional bound of interest in the study of polynomial hulls. In \cite{IzSt18}, Izzo and Stout show that any surface can be embedded in $\C^3$ so as to have nonpolynomially convex image with no analytic disk in its hull. They then pose the following question. {\em For a fixed $m\geq 3$, what is the smallest $n$ such that every compact $m$-dimensional smooth manifold can be smoothly
embedded into $\Cn$ as some $\Sigma$ with $\widehat\Sigma\setminus\Sigma$ nonempty but
containing no analytic disk, i.e., there is no nonconstant holomorphic map from the unit disk into $\widehat\Sigma\setminus\Sigma$?} In \cite{ArWo17}, it is shown that if the embedding is also required to be totally real, then the optimal value of $n$ is $\lfloor{3m/2\rfloor}$, for any $m\geq 2$. In \cite{Iz18}, it is shown that the constructions in \cite{IzSt18} and \cite{ArWo17} can be done so that the rational and polynomial hulls of the embeddings coincide. In our next result, we show that the answer to the original question is strictly less than $\lfloor{3m/2\rfloor}$ for even-dimensional manifolds. 

\begin{Prop}\label{thm_anstrhull} For any $2k$-dimensional compact manifold $M$, there is a smooth embedding of $M$ into $\C^{3k-1}$ with image $\Sigma$ so that $\widehat{\Sigma}\setminus \Sigma$ is nonempty but contains no analytic disk, and $\widehat \Sigma=h_r(\Sigma)$, the rationally convex hull of $\Sigma$.	
\end{Prop}
Our results show that, in spite of the presence of CR-singularities, $2k$-submanifolds  in $\C^{3k-1}$ behave like totally real submanifolds of $\Cn$ when it comes to polynomial convexity and hulls. A finer analysis of the Beloshapka-Coffman normal form allows us to recover, albeit with slight modifications, some more properties satisfied by totally real submanifolds of $\Cn$. For instance, every polynomially convex compact set $K\subset\Cn$ is the zero locus of a smooth nonnegative plurisubharmonic (p.s.h.) function on $\Cn$ that is strictly p.s.h. outside of $K$ (see \cite[Theorem 1.3.8]{St07}). If we additionally assume that $K$ is a totally real submanifold, then the function can be chosen to be strictly p.s.h. everywhere. This follows from the fact that for any totally real submanifold $M\subset\Cn$, the square-distance function, $\operatorname{dist}^2(\cdot, M)$, gives a locally defined strictly p.s.h. defining function for $M$. This local defining function also grants a symplectic property to rationally convex totally real submanifolds: any such $M^m\subset\Cn$ is {\em Lagrangian} (or {\em isotropic} if $m<n$) with respect to some K{\" a}hler form $\om$ on $\Cn$, i.e., $\iota^*\om=0$, where  $\iota: M\hookrightarrow\Cn$ is the inclusion map (see Duval-Sibony \cite{DuSi95}). We obtain analogous results for $2k$-manifolds with generic CR-singularities. We note that the construction of p.s.h. defining functions with additional properties is of independent interest in the literature (\cite{Sl04}), and is related to the existence of regular Stein neighbourhood bases. 

\begin{theorem}\label{thm_deffn} Let $M\subset\C^{3k-1}$ be a $2k$-dimensional smooth compact connected submanifold that is totally real except on a finite set of generic CR-singularities, say $S$. Then, 
	\begin{enumerate}
\item $M=\rho^{-1}(0)$ where $\rho$ is a smooth nonnegative function on some neighbourhood $U$ of $M$ and is strictly p.s.h. on $U\setminus S$. 
\item If $M$ is rationally convex, then $M$ is isotropic with respect to $dd^c\varphi$, for some p.s.h. function $\varphi$ on $\C^{3k-1}$ that is strictly p.s.h. on $\C^{3k-1}\setminus S$. 
\item If $M$ is polynomially convex, then $M=\rho^{-1}(0)$ where $\rho$ is a smooth nonnegative function on $\C^{3k-1}$ and is strictly p.s.h. on $\C^{3k-1}\setminus S$. 
\end{enumerate}
\end{theorem}

Statement $(2)$ above yields a variation of the Gromov-Lees theorem~\cite{AuLaPo94}, which in turn is an application of
Gromov's $h$-principle. The Gromov-Lees theorem says that a compact $n$-dimensional manifold $M$ admits a 
Lagrangian immersion into $(\Cn,\om\std)$ if and only if its complexified tangent 
bundle is trivializable. This is the same topological condition that completely characterizes the totally real immersability of a manifold 
$M$ in $\Cn$ (see~\cite[Prop. 9.1.4]{Fo11}). Subcritical versions of these results imply that any compact $m$-dimensional manifold
admits an isotropic embedding into $(\Cn,\om\std)$ for $n\ge \floor{\frac{3m}{2}}$. Furthermore, there exist $m$-dimensional 
manifolds that do not admit such embeddings when $n< \floor{\frac{3m}{2}}$; see~\cite{GuSh17} for details. Despite this fact, 
our result shows that if $m$ is even, any $m$-dimensional $M$ can be embedded as an isotropic submanifold in $\C^{\floor{\frac{3m}{2}}-1}$ with respect to some {\it degenerate} K\"ahler form.  The
proof does not however rely on the $h$-principle.  Instead, we use a characterization of rational convexity established in Duval-Sibony~\cite{DuSi95}.

\noindent {\bf Acknowledgments.} We would like to thank Alexander Izzo for his helpful comments on an earlier version of this paper. In particular, he observed the relavance of our approach to the question of hulls with no analytic disks, which is now addressed in Proposition~\ref{thm_anstrhull}.
\section{Background material}\label{sec_backmat}
The reader can refer to this section for the notation, terminology and definitions used in this paper. We begin with some notation. 

\begin{itemize} 
\item $D_z(r)$ and $\bar D_z(r)$ denote the open and closed disks, respectively, of radius $r$ centred at $z$ in $\C$. 
\item $B_p(r)$ and $\bar B_p(r)$ denote the open and closed Euclidean balls, respectively, of radius $r$ centred at $p$ in $\Cn$, $n>1$. 
\item $O$ denotes the origin in $\C^n$ (the `$n$' will be clear from the context).  
\item $Z=(z,w_1,...,w_{2k-2},\zeta_1,...,\zeta_k)$ denotes the complex coordinates in $\C^{3k-1}$, where  
		\begin{align*}
		&z=x+iy,\\
		&w_\tau=u_\tau+iv_\tau,\quad 1\leq \tau\leq 2k-2,\\
		&\zeta_\sigma=\xi_\sigma+i\eta_\sigma,\quad  1\leq \sigma\leq k ,
	\end{align*}
is the decomposition of the coordinates into their real and imaginary parts. 
\item $Z'=(z,w_1,...,w_{2k-2},w)$ denotes the complex coordinates in $	\C^{2k}$. 
\item $\xi^*$ is the conjugate transpose of the vector $\xi\in\Cn$ (viewed as a  matrix).
\item $\JacC f(Z)$ denotes the complex Jacobian at $Z$ of the map $f:\C^{3k-1}\rightarrow\C^m$.
\item $\Chess f(Z)$ denotes the complex Hessian of $f:\C^{3k-1}\rightarrow\rl$ at $Z$.
\item For any compact set $X\subset\Cn$, $\cont(X)$ is the algebra of complex-valued continuous functions on $X$, and $\mathcal{P}(X)$ is the closure in $\cont(X)$ of the subalgebra generated by all the holomorphic polynomials restricted to $X$. 
\end{itemize} 

A necessary condition for a set $X\subset\Cn$ to satisfy $\mathcal{P}(X)=\cont(X)$ is that it must  coincide with its {\em polynomially convex hull}
	\bes
		\widehat X:=\left\{x\in\Cn:|P(x)|\leq \sup_{z\in X}|P(z)|, \text{ for all polynomials $P$ in}\ \Cn\right\}.
	\ees
If $X=\widehat X$, we say that $X$ is {\it polynomially convex}. If we replace polynomials in the above definition by rational functions in $\Cn$ with no poles on $X$, then we obtain the related notions of {\em rationally convex hulls} and {\em rational convexity}. A sufficient condition for a polynomially convex submanifold $M\subset\Cn$ to satisfy $\cont(M)=\mathcal P(M)$ is that $M$ be {\it totally real}, i.e., $T_pM\cap iT_p(M)=\{0\}$ for all $p\in M$, where $T_pM$ denotes the real tangent space of $M$ at $p$. Thus, $\cont(M)=\mathcal{P}(M)$ if $M$ is a totally real and polynomially convex submanifold of $\Cn$. 

As discussed earlier, it is not always possible to arrange $M\subset \Cn$ to be totally real everywhere. Given a point $p\in M$, let $H_pM$ denote the maximal complex-linear subspace of $T_pM$. A point $p\in M$ is called a {\em CR-singularity of $M$} if $\dim_\C(H_pM)\geq 1$. As a consequence of Thom's transversality theorem (see, e.g.,~\cite{GoGu73}), the set $S$ of CR-singularities of a generically embedded $M\subset\Cn$ is either empty or is a smooth submanifold of codimension $2(n-m)+2$ in $M$, see Domrin~\cite{Do95} for more details. 
Since $M$ is always locally polynomially convex near its totally real points, we must study the convexity properties of $M$ near $S$. The situation is nontrivial even when $S$ is a discrete set, i.e., when $m=2k$ and $n=3k-1$.      

When $k=1$ (or $m=n=2$), the only possible CR-singularities are complex points. These were classified by Bishop in \cite{Bi65} as follows. Given an isolated nondegenerate complex point $p$ of a surface $M$, one can find local holomorphic coordinates in which $M$ can be written as
	\bes
		w=\begin{cases}
				\frac{\alpha}{2} z\zbar+\frac{1}{4}(z^2+\zbar^2)+o(|z|^2),& \ \text{if}\ 0\leq \alpha<\infty,\, \\
				z\zbar+o(|z|^2),&\ \text{if}\ \alpha=\infty.
			\end{cases}
	\ees
Depending on whether $\alpha\in[0,1)$, $\alpha=1$ or $\alpha\in[1,\infty]$, $p$ is said to be a hyperbolic, parabolic or elliptic complex point, respectively. Parabolic points are not generic, and have varying local convexity properties (see \cite{Wi95} and \cite{Jo97}). Although, elliptic and hyperbolic points are both stable under small $\cont^2$-deformations, a surface is locally polynomially convex only near its hyperbolic complex points. In \cite{Sl04}, Slapar proves a (possibly stronger) result for {\em flat} hyperbolic points (\cite{FoSt91}), i.e., when local holomorphic coordinates can be chosen so that $\ima o(|z|^2)\equiv 0$, i.e., $M$ is locally contained in $\C\times\rl$. It is shown that, near a flat hyperbolic $p$, $M$ is the zero set of a nonnegative function that is strictly p.s.h. in its domain except at $p$. 
 
The case $k>1$ (i.e., $m=2k$ and $n=3k-1$) is qualitatively different because of higher codimension ($m<n$). Here, stable CR-singularities do not show diverse behaviour in this regard. In fact, it suffices to understand one special model to answer this question. We call this model the {\em Beloshapka-Coffman normal form} and it is given by the manifold
	\bea\label{eq_BCform}
		\bc_k:=
\left\{Z\in\C^{3k-1}:
\begin{aligned}
&v_\tau=0,\quad 1\leq \tau\leq 2k-2,\\
&\zeta_{1}=|z|^2+\zbar(u_1+iu_2), \\
&\zeta_\sigma=\zbar(u_{2\sigma-1}+iu_{2\sigma}),\quad 2\leq \sigma\leq k-1, \\
&\zeta_k=\zbar^2
\end{aligned}
\right\}.
	\eea
Note that $\dim \bc_k=2k$ and it has an isolated CR-singularity at the origin. In~\cite{Be97} and \cite{Co97}, Beloshapka ($k=2$) and Coffman ($k\geq 2$) showed that a nondegenerate CR-singularity $p$ of a $2k$-dimensional submanifold $M$ of $\C^{3k-1}$ is locally {\em formally} equivalent to $\bc_k$ at the origin. The nondegeneracy conditions appearing in their work are the full-rank conditions on matrices involving the second-order derivatives of the graphing functions of $M$ at $p$. Any isolated CR-singular point can, thus, be made nondegenerate with the help of a small $\cont^\ell$-perturbation, $\ell\geq 2$.  In~\cite{Co06}, Coffman further proved that if $M$ is also real analytic in a neighbourhood of $p$, then there is a local normalizing transformation that is given by a convergent power series. Since any smooth $M$ near a nondegenerate CR-singularity $p$ can be made real analytic after a small $\cont^\ell$-perturbation, we will only concern ourselves with real analytic nondegenerate CR-singularities. These will be referred to as {\em generic CR-singularities} in this paper. We rely on the fact that any $M$ at a generic CR-singularity $p$ is locally biholomorphic to $\bc_k$ at $O$. In Section \ref{sec_deffn}, we show more: near $O$, $\bc_k$ is the zero set of a nonnegative function that is strictly p.s.h. everywhere except at $O$. This shows that generic CR-singularities of $2k$-manifolds in $\C^{3k-1}$ behave like flat hyperbolic complex points.

We now note (and prove) a well-known fact that will be used multiple times in this paper. 

\begin{lemma}\label{lem_union}
	Let $K\subset \Cn$ be a polynomially convex compact set and $p_1,...,p_\ell\in\Cn\setminus K$. Then, there exist $r_1,...,r_\ell>0$, so that
		\bes		
			K\cup\bigcup_{j=1}^\ell\bar B_{p_j}(r_j')
		\ees
	is polynomially convex for all $r_j'\leq r_j$, $j=1,...,\ell$.
\end{lemma}
\begin{proof} We prove the claim by induction on $\ell$. Suppose $\ell=1$. Since $K$ is polynomially convex and $p_1\notin K$, there is a polynomial $Q:\Cn\rightarrow\C$ so that $|Q(p_1)|>\sup_{K}|Q|$. Thus, we may choose $s>\sup_K|Q|$ and $t>0$ such that $\bar D_0(s)$ and $\bar D_{Q(p_1)}(t)$ are disjoint in $\C$. Let $r_1>0$ be small enough so that $\bar B_{p_1}(r_1)\subset Q^{-1}(D_{Q(p)}(t))$. Then, since $Q(K)$ and $Q\left(\bar B_{p_1}(r_1')\right)$ lie in disjoint disks in $\C$ for all $r_1'\leq r_1$, by Kallin's lemma (see\cite{St07}), $K\cup \bar B_{p_1}(r_1')$ is polynomially convex for all $r_1'\leq r_1$. Now, suppose the claim holds for $\ell=m-1$, and let $p_1,...,p_m\notin K$. The induction hypothesis gives $r_1,...,r_{m-1}>0$ so that for any $r_j'\leq r_j$, $j=1,...,m-1$, $K'=K\cup\bar B_{p_1}(r_1')\cup\cdots\cup\bar B_{p_{m-1}}(r_{m-1}')$ is polynomially convex. We may shrink the $r_j$'s to ensure that $p_m\notin K'$. Now, repeating the proof for the case $\ell=1$ with $K=K'$ and $p_1=p_\ell$, we obtain the claim for $\ell=m$, and hence for all $\ell\in\N$. 
 \end{proof}

\begin{rmk} The above proof actually gives a stronger conclusion:   there exist $r_1,...,r_\ell>0$, so that $K\cup M_1\cup\cdots \cup M_\ell$ is polynomially convex, for any polynomially convex compacts $M_j\subset\bar B_{p_j}(r_j)$, $j=1,...,\ell$.
\end{rmk}

\section{Proof of Theorem \ref{thm_main} and Corollary \ref{cor_dens}}
\label{sec_thm1}	

We first begin with a result on the local polynomial convexity of the Beloshapka-Coffmam normal form, which is of independent interest. 

\begin{lemma}\label{lem_BC} The manifold $\bc_k\in\C^{3k-1}$ in \eqref{eq_BCform} is locally polynomially convex at $O$. 
\end{lemma}
\begin{proof} We recall the following criterion (an iterated version of Theorem 1.2.16 from~\cite{St07}). {\em If $X\subset\Cn$ is a compact subset and if $G:X\rightarrow \rl^m$ is a map whose components are in $\mathcal{P}(X)$, then $X$ is polynomially convex if and only if $G^{-1}(\bt)$ is polynomially convex for each $\bt\in\rl^m$.} Now, choose the restriction to $\bc_k$ of $G:\C^{3k-1}\rightarrow\C^{2k-2}$ that maps $Z\mapsto(w_1,..,w_{2k-2})$. Then, since the subalgebra generated by $z$ and $\zbar^2$ in $\cont(\clD_\eps)$ coincides with $\cont(\clD_\eps)$ (see~\cite{MI76}), we have that every fibre of $G$ is polynomially convex. Hence, by the criterion stated above, $\bc_k$ is locally polynomially convex at $O$. 
\end{proof}

\noindent {\em Proof of Theorem~\ref{thm_main}}.
Let $\iota:M\hookrightarrow\C^{3k-1}$ be the inclusion map of a smooth $2k$-dimensional submanifold $M\subset\C^{3k-1}$. Fix $s\geq 2$. By Thom's Transversality Theorem, there exists a $\cont^s$-small perturbation $j$ of $\iota$ such that $j(M)$ is smooth and totally real except at a finite number of CR-singular points (see \cite[Section~1]{Do95} for details). Without loss of generality, we may further assume that $j(M)$ has generic CR-singular points (see end of Section~\ref{sec_backmat}). Let $p_1,...,p_\ell$ denote the CR-singularities of $j(M)$. Since, for each $j$, $(M,p_j)$ and $(\bc_k,O)$ are locally biholomorphic, Lemma~\ref{lem_BC} shows that small enough neighbourhoods of $p_j$ in $M$ are polynomially convex.  Applying Lemma~\ref{lem_union} and the subsequent remark to $p_1,...,p_\ell$ (and $K=\emptyset$), we obtain opens sets $W_1,...,W_\ell\subset M$ containing $p_1,...,p_\ell$, respectively, so that the closure of $W=W_1\cup\cdots\cup W_\ell$ is polynomially convex and $j(M)\setminus W$ is a compact submanifold of $\C^{3k-1}$ with boundary. Since, $j(M)\setminus W$ is totally real, we can now apply the following result due to Arosio-Wold (see~\cite[Theorem~1.4]{ArWo17}). {\em Let $N$ be a compact smooth manifold (possibly with boundary) of dimension $d<n$ and let $f:N\rightarrow\Cn$ be a totally real $\cont^\infty$-embedding. Let $K\subset\Cn$ be a compact polynomially convex set. Then for all $s\geq 1$ and for all $\eps>0$, there exists a totally real $\cont^\infty$-embedding $f_\eps:N\rightarrow\Cn$ such that
	\begin{enumerate}
		\item $||f-f_\eps||_{\cont^s(N)}<\eps$ 
		\item $f_\eps=f$ near $f^{-1}(K)$, and
		\item $\widehat{K\cup f_\eps(N)}=K\cup f_\eps (N)$. 
	 \end{enumerate}} 
In our case, $N=M\setminus j^{-1}(W)$, $f=j|_N$ and $K=\overline W$. Let $\eps>0$ be arbitrary. Set $M'=f_\eps(M\setminus j^{-1}(W))\cup W$ to obtain a polynomially convex perturbation of $M$ that is totally real except at $p_1,...,p_\ell$.
\qed

\medskip

\noindent {\em Proof of Corollary~\ref{cor_dens}}.
Let $M$ be a compact $2k$-dimensional abstract manifold without boundary. By Theorem~\ref{thm_main}, there exists a $\cont^\infty$-smooth embedding $F=(f_1,...,f_{3k-1}):M\rightarrow\C^{3k-1}$ such that $F(M)$ is polynomially convex and totally real outside a finite set $S\subset F(M)$. For any compact set $X\subset\Cn$, we let 
	\bes
		\hol(X)=\{f|_X:f\ \text{is holomorphic in some open neighbourhood of}\ X\}.
	\ees
Note that $X:=F(M)$ and $X_0:=S$ satisfy the hypothesis of the following result due to O'Farrel-Preskenis-Walsch ((see~\cite{OFPrWa84}; also see~\cite{St07})). {\em Let $X$ be a compact holomorphically convex set in $\Cn$, and let $X_0$ be a closed subset of $X$ for which $X\setminus X_0$ is a totally real subset of the manifold $\Cn\setminus X_0$. A function $f\in\cont(X)$ can be approximated uniformly on $X$ by functions holomorphic on a neighbourhood of $X$ if and only if $f|_{X_0}$ can be approximated uniformly on $X_0$ by functions holomorphic on $X$.} 

Hence, $\overline{\hol(F(M))}=\cont(F(M))$. Further, by the Oka-Weil theorem for polynomially convex sets, we have that $\mathcal{P}(X)=\overline{\hol(F(M))}$. Thus, $\{P\circ F:P\ \text{is a holomorphic polynomial on}\ \C^{3k-1}\}$ is dense in $\cont(M)$. In other words, $\{f_1,...,f_{3k-1}\}$ is a PD-basis of $\cont(M)$. 

Now, if $M$ is a manifold with boundary, Theorem~\ref{thm_main} guarantees a smooth embedding $F:M\rightarrow\C^{3k-1}$ such that $F(M)$ is totally real and polynomially convex (see the comment following the statement of Theorem~\ref{thm_main}). We fix an $\ell\geq 0$, a $g\in\cont^\ell(M)$ and an arbitrary $\eps>0$. Let $\wt\eps=C\eps$, where $C$ is a constant to be determined later. Since $F(M)$ is totally real, a result due to Range-Siu (see Theorem 1 in \cite{RaSi74}; although not explicitly stated, the result therein works for compact manifolds with or without boundary) grants the existence of a neighbourhood $U$ of $F(M)$ and a $h\in\hol(U)$ such that
	\be\label{eq_c-h}
		||g-h||_{\cont^\ell(F(M))}<\wt\eps.
	\ee
Due to the polynomial convexity of $F(M)$, we can find a neighbourhood $V\Subset U$ of $F(M)$, such that $\overline V$ is polynomially convex. By the Oka-Weil approximation theorem, there is a polynomial $P$ on $\Cn$ such that  
	\be\label{eq_h-r_sup}
		||h-P||_{\cont(\overline V)}< \wt\eps.
	\ee
As $F(M)$ is compact, there is an $r>0$ such that $B_x(r)\subseteq V$ for all $x\in F(M)$. We fix an $x\in F(M)$. As $h-P\in\hol(B_x(r))$, we can combine Cauchy estimates and \eqref{eq_h-r_sup} to obtain
	\bea 
		\left|h^{(j)}(x)-P^{(j)}(x)\right|
			&\leq&\frac{j!}{r^j}\sup_{y\in B_x(r)}|h(y)-P(y)| 
			\leq \frac{j!}{r^j}\wt\eps,		\label{eq_h-r_ck}								
	\eea
for any $j\in\pnat$. So, we obtain from \eqref{eq_c-h} and \eqref{eq_h-r_ck} that
	\beas
		||g-P||_{\cont^\ell(F(M))}&\leq &||g-h||_{\cont^\ell(F(M))}+||h-P||_{\cont^\ell(F(M))}\\
		&=&||g-h||_{\cont^\ell(F(M))}+\sum_{j=0}^k ||h^{(j)}-P^{(j)}||_{\cont(F(M))}\\
		&<&\wt \eps\left(1+\sum_{j=0}^k \frac{j!}{r^j}\right)
		=C\eps\left(1+\sum_{j=0}^k \frac{j!}{r^j}\right).
	\eeas
Setting $C=\left(1+\sum_{j=0}^k \frac{j!}{r^j}\right)^{-1}$, we obtain that $	||g-P||_{\cont^\ell(M)}<\eps$. Since $\eps$ and $g$ were chosen arbitrarily, and $C$ is independent of $\eps$, we conclude that polynomials are dense in the space of $\cont^\ell$-smooth functions on $F(M)$ in the $\cont^\ell$-norm. Thus, the components of $F$ form a PD-basis of $\cont^\ell(M)$, for every $\ell\geq 0$.
 \qed

\begin{rmk} The following statement is implicit in the above proof. If we view a PD-basis of $M$ as a map from $M$ to $\C^{3k-1}$, then the set of all PD-bases of $\cont(M)$ is dense in $\cont(M;\C^{3k-1})$. This follows from the fact that smooth embeddings of $M$ into $\C^{3k-1}$ are dense in $\cont(M;\C^{3k-1})$ for $k>1$. 
\end{rmk}

\section{Proof of Proposition~\ref{thm_anstrhull}}
We first state a theorem due to Alexander which is used both by Izzo-Stout (in \cite{IzSt18}) and Arosio-Wold (in \cite{ArWo17}) in their respective constructions of hulls with no analytic structure. 

\begin{theorem}[Alexander, \cite{Al98}]  The standard torus $\mathbb{T}^2=\{(e^{i\theta},e^{i\psi}):\theta,\psi\in\rl\}$ in $\CC$ contains a compact subset $E$ such that $\widehat E\setminus E$ is nonempty but contains no analytic disk. Such a set can be chosen in any neighbourhood of the diagonal of $\mathbb{T}^2$.  
\end{theorem}

Let $A$ be a tubular neighbourhood of the diagonal in $\mathbb{T}^2$ and extend it to a smooth totally real $2k$-dimensional submanifold of $\C^{3k-1}$ as follows:
	\beas
			U=
\left\{Z\in\C^{3k-1}:
\begin{aligned}
&(z,w_1)\in A,\\
& |\rea(w_t)|<\eps,\ \ima(w_t)=0,\quad 2\leq t\leq 2k-2,\\
& |\rea(\zeta_1)|<\eps,\ \ima(\zeta_1)=0,\ \zt_2=\cdots=\zeta_k=0
\end{aligned}
\right\}.
	\eeas
Let $E$ denote an Alexander set in $A$. We abuse notation and denote $E\times\{0\}^{3k-3}\subset U$ by $E$. Since a generic embedding of $M$ into $\C^{3k-1}$ is totally real except for finitely many generic CR-singularities, we may consider a smooth copy of $M$ in $\C^{3k-1}$ (also denoted by $M$) that contains $U$ in a small $2k$-dimensional ball in its interior and has generic CR-singularities $p_1,...,p_\ell\in M\setminus U$ that are disjoint from $\widehat E$. Since $M$ is locally polynomially convex at $p_1,...,p_\ell$, Lemma \ref{lem_union} and the subsequent remark show that there exists a neighbourhood $W$ of the set $\{p_1,...,p_\ell\}$ in $M$ so that $\widehat E\cup \overline W$ is polynomially convex, and $M\setminus W$ is a totally real smooth submanifold with boundary. We now apply the Arosio-Wold perturbation result stated in Section \ref{sec_thm1} to $N=M\setminus W$, $K=\widehat E\cup\overline W$ and $f$, the inclusion map of $M\setminus W$, to obtain a smooth embedding of $M$ into $\C^{3k-1}$ whose image $\Sigma$ contains $E\cup W$, and $\widehat{\Sigma\cup \widehat E}= \Sigma\cup\widehat E$. Thus, $\widehat{\Sigma}=\widehat{\Sigma\cup E}= \Sigma\cup\widehat E$. Now, if $\widehat E$ was contained in $\Sigma$, then $\Sigma$ would be a polynomially convex manifold that is totally real except at generic CR-singularities. We have shown in the proof of Corollary \ref{cor_dens} that any subset $T$ of such a manifold has the property that $\mathcal{P}(T)=\cont(T)$, and thus is polynomially convex. This contradicts the fact that $ E\subset \Sigma$ is not polynomially convex. Thus, 
$\widehat\Sigma\setminus\Sigma$ is nonempty but contains no analytic disk. 

To show that $\widehat{\Sigma}=h_r(\Sigma)$, we use Izzo's argument from \cite[Section 3]{Iz18}. He shows that the set $E$ satisfies the generalized argument principle, i.e., if $p$ is a polynomial that has a continuous logarithm on $E$, then $0\notin p(E)$. Then, we use the following result due to Stolzenberg (\cite{St63}). {\em If $X\subseteq Y\subset\Cn$ are compact sets such that $X$ satisfies the generalized argument principle and the first {\v C}ech cohomology group $\check{\mathrm{H}}^1(Y,\mathbb{Z})$ vanishes, then $\widehat X\subset h_r(Y)$.} Since $E$ is contained in a (contractible) ball $Y$ in $\Sigma$, $\widehat E\subseteq h_r(Y)\subset h_r(\Sigma)$. So, $\widehat{\Sigma}= \Sigma\cup\widehat{E}\subseteq h_r(\Sigma)$. Thus, the two hulls coincide, as claimed. 
\section{Proof of Theorem~\ref{thm_deffn}}\label{sec_deffn}

The polynomial convexity established in Lemma~\ref{lem_BC} allows us to write $\bc_k$ near $O$ as the zero locus of some nonnegative p.s.h. function that is strictly p.s.h. away from $\bc_k$. In order to obtain Theorem~\ref{thm_deffn}, we need an improved version of this fact, which we establish in the following technical proposition.  
\begin{Prop}\label{prop_Slapar} Let $k>1$. For a given $r>0$ small enough, there exists a smooth p.s.h. function $\psi:\C^{3k-1}\rightarrow\rl$ such that 
\renewcommand{\theenumi}{\alph{enumi}}
	\begin{enumerate}
		\item $\{\psi=0\}=\bc_k\cap \bar B_O(r)$,
		\item $\psi>0$ on $\C^{3k-1}\setminus (\bc_k\cap \bar B_O(r))$, and
		\item $\psi$ is strictly p.s.h. on $\C^{3k-1}\setminus\{O\}$. 
	\end{enumerate}
\end{Prop}
\begin{proof} We first construct a $\wt\psi$ that has all the desired properties of $\psi$ but is only defined locally  near $O\in\bc_k\subset\C^{3k-1}$. We work with an auxiliary family of $2k$-manifolds in $\C^{2k}$. Let $\alpha<1$. Set
	\bes
		S_\alpha=\left\{Z'\in\C^{2k}:
			\begin{aligned}
				&\ima(w_1)=\cdots=\ima(w_{2k-2})=0,\\
					&w=\frac{\alpha}{2}|z|^2+\frac{1}{4}(z^2+\zbar^2)
			\end{aligned}
		\right\}.
	\ees
Each slice $S_\alpha\cap\{Z'\in\C^{2k}:(w_1,...,w_{2k-2})=(s_1,...,s_{2k-2})\}$, where $(s_1,...,s_{2k-2})\in\rl^{2k-2}$, is a totally real surface with an isolated flat hyperbolic complex point at the origin in $\CC_{z,w}$. These have been studied by Slapar in~\cite{Sl04}. A slight modification of his construction yields the following key ingredient of our proof.  

\begin{lemma}\label{lem_Slaparfn}
For each $\alpha<0.46$, there is a neighbourhood $V_\alpha$ of the origin in $\C^{2k}$ and a smooth p.s.h. function $\rho_\alpha:V_\alpha\mapsto \rl$ such that
\renewcommand\labelitemi{$\boldsymbol *$}
	\begin{itemize}
		\item  $\{\rho_\alpha=0\}=S_\alpha\cap V_\alpha$,
		\item  $\rho_\alpha> 0$ on $V_\alpha\setminus S_\alpha$, and
		\item  $\rho_\alpha$ is strictly p.s.h. on $V_\alpha\setminus Y$, where 
				\be\label{eq_sing}
					Y:=\{Z'\in \C^{2k}:z=\ima w_1=\cdots=\ima w_{2k-2}=w=0\}.
				\ee
	\end{itemize}
\end{lemma}

We relegate the proof of this lemma to the appendix (see Section~\ref{sec_append}). To continue with the proof of Proposition~\ref{prop_Slapar}, we produce holomorphic maps that send the Beloshapka-Coffman normal form $\bc_k$ into $S_\alpha$. These allow us to pull back $\rho_\alpha$ to $\C^{3k-1}$ (locally near $O$) to give p.s.h. functions that vanish on $\bc_k$. For this, let $f_\alpha:\C^{3k-1}\rightarrow\C^{2k}$ be the map 
	\begin{align*}
		Z\mapsto \left(z+\tfrac{\alpha w_1}{\alpha+1}
				-\tfrac{i\alpha w_2}{1-\alpha}, w_1,...,w_{2k-2},
		\tfrac{\alpha}{2}\zeta_1+\tfrac{\zeta_k}{4}+\tfrac{z^2}{4}
		+\tfrac{\alpha}{2}z(w_1-iw_2)+\tfrac{\alpha^2 w_1^2}{2\alpha+2}-\tfrac{\alpha^2 w_2^2}{2-2\alpha}\right).
	\end{align*}
For $1\leq \sigma\leq {k-1}$, let $f^\sigma_\alpha:\C^{3k-1}\rightarrow\C^{2k}$ be the map given by
	\bes
		f^\sigma_\alpha=f_\alpha\circ F^\sigma,
	\ees
where $F^\sigma:\C^{3k-1}\rightarrow\C^{3k-1}$ is the automorphism
	
	\beas
\begin{pmatrix} z,w_1,\cdots, w_{2k-2}, \zeta_1, \cdots, \zeta_k\end{pmatrix}
	\longmapsto		
		\begin{pmatrix}z,
    		\dfrac{w_1+w_{2\sigma-1}}{2},
    		\dfrac{w_2+w_{2\sigma}}{2},
		 	w_3,
			\cdots,
			w_{2k-2},
			\dfrac{\zeta_1+\zeta_\sigma}{2},
			\zeta_2,
			\cdots,
			\zeta_k \end{pmatrix}.	
	\eeas
Each $f_\alpha^\sigma$ is holomorphic on $\C^{3k-1}$ and has the following properties.
	\begin{itemize}
		\item $(f_\alpha^\sigma)^{-1}(S_\alpha)=M_\alpha^\sigma$, where
					\begin{align*}
					\qquad\qquad	M_\alpha^\sigma=\left\{Z\in\C^{3k-1}:
							\begin{aligned}
								&\alpha\big[\zeta_1+\zeta_\sigma
									-\zbar(2z+w_1+w_{2\sigma-1}+i(w_2+w_{2\sigma}))\big]
										+\zeta_k-\zbar^2=0,\\
								&\ima w_1=\cdots=\ima w_{2k-2}=0
							\end{aligned}
						\right\}.
					\end{align*}
		\item $(f_\alpha^\sigma)^{-1}(Y)=X_\alpha^\sigma$, where
					\begin{align*}
			 			X_\alpha^\sigma=\left\{Z\in\C^{3k-1}:
							\begin{aligned}
								&z+\frac{\alpha (w_1+w_{2\sigma-1})}{2\alpha+2}
									-\frac{i\alpha (w_2+w_{2\sigma})}{2-2\alpha}=0,\\
								&\alpha(\zeta_1 +\zeta_\sigma)+\zeta_k+\zbar^2=0,\\
								&\ima w_1=\cdots=\ima w_{2k-2}=0
							\end{aligned}
						 \right\}.
					\end{align*}
		\item $\ker\JacC(f_\alpha^\sigma)(Z)=\left\{\left(
				\underbrace{0,...,0}_{2k-1},
					\zeta_1,...,\zeta_{k-1},-\alpha(\zeta_1+\zeta_\sigma)\right)
						:(\zeta_1,...,\zeta_{k-1})\in\C^{k-1}\right\}$.
	\end{itemize}

Next, let $\psi_\alpha^\sigma:=\rho_\alpha\circ f_\alpha^\sigma$ on $U_\alpha^\sigma:=(f_\alpha^\sigma)^{-1}(V_\alpha)$, where $\rho_\alpha$ and $V_\alpha$ are as in Lemma~\ref{lem_Slaparfn}. Then, owing to the properties of $\rho_\alpha$ and $f_\alpha^\sigma$, we have that $\psi_\alpha^\sigma$ is a p.s.h. function on $U_\alpha^\sigma$, satisfying the following properties (compare with the required properties (a)-(c)).
\renewcommand{\theenumi}{\alph{enumi}'}
	\begin{enumerate}
		\item $\{\psi_\alpha^\sigma=0\}=M_\alpha^\sigma\cap U_\alpha^\sigma$,
		\item $\psi_\alpha^\sigma>0$ on 
				$U_\alpha^\sigma\setminus M_\alpha^\sigma$, and 
		\item $ \xi^* \cdot 	\Chess\psi_\alpha^\sigma(Z) \cdot \xi>0$,
				 when $Z\in U_\alpha^\sigma\setminus X_\alpha^\sigma$
				 and $\xi\in\C^{3k-1}\setminus \ker\JacC(f_\alpha^\sigma)(Z)$.
	\end{enumerate} 
As $\bc_k\subsetneq M_\alpha^\sigma$, we need to `correct' $\psi_\alpha^\sigma$. For this, let 
	\bes
		g(Z)=|\zeta_k-\zbar^2|^2
			+\sum_{\sigma=2}^{k-1}|\zeta_\sigma
				-\zbar(\overline{w_{2\sigma-1}}+i\:\overline{w_{2\sigma}})|^2. 
	\ees
Since $M_\alpha^\sigma\cap g^{-1}(0)=\bc_k$, and $\xi^*\cdot \Chess g(Z)\cdot\xi>0$ for any $Z\in\C^{3k-1}$ and any nonzero $\xi\in \ker\JacC(f_\alpha^\sigma)(Z)$, we have that each $g+\psi_\alpha^\sigma$ is a p.s.h. function on $U_\alpha^\sigma$ satisfying properties $(a)$, $(b)$ and 
	\begin{itemize}
		\item [(c'')] $\psi_\alpha^\sigma+g$ is strictly p.s.h. on 
							$U_\alpha^\sigma\setminus X_\alpha^\sigma$.
	\end{itemize}
Finally, to obtain property (c), we observe that
 	\bes	
		\bigcap _{\sigma=1}^{k-1}\left(X_\alpha^{\sigma}
			\cap X_\beta^{\sigma}\right) 
			=\{O\}
	\ees
when $\alpha\neq\beta$. Thus, choosing $\alpha=1/4$ and $\beta=1/3$, we have that
		\be\label{eq_defpsi}
			\wt\psi:=g+\sum_{\sigma=1}^{k-1}\left(\psi_{_{1/4}}^\sigma
				+\psi_{_{1/3}}^\sigma\right)
		\ee
is a p.s.h. function on $U:=\bigcap _{\sigma=1}^{k-1}\left(U_{_{1/4}}^{\sigma}\cap U_{_{1/3}}^{\sigma}\right)$ satisfying the local versions of properties (a)-(c).

To complete the proof of Proposition~\ref{prop_Slapar}, we extend $\wt\psi$ to $\C^{3k-1}$. Choose $r>0$ small enough so that $B=\bar B_O(r)\subset U$.  As $B$ is polynomially convex, there is a smooth nonnegative p.s.h. function $\sigma$ on $\C^{3k-1}$ such that $B=\sigma^{-1}(0)$ and $\sigma$ is strictly p.s.h. on $\C^{3k-1}\setminus B$. Choose closed balls $B'$ and $B''$ so that $B\subset B'\subset B'' \subset U$. Let $\chi$ be a smooth function on $\C^{3k-1}$ that is $1$ on an open set containing $B$, $0$ on an open set containing $\C^{3k-1}\setminus U$, and always between $0$ and $1$. Then, for large enough $C>0$, $\psi=\chi\wt\psi+C\sigma$ has the desired properties.
\end{proof}

\begin{rmk}
The above result shows that if $p$ is a generic CR-singularity of a $2k$-manifold $M\subset\C^{3k-1}$, then any polynomially convex neighbourhood $N\subset M$ of $p$ is stable in the following sense: for small enough perturbations $\phi$ that are identity close to $p$, $\phi(N)$ is a polynomially convex neighbourhood of $p=\phi(p)$.
\end{rmk}

We now have the main ingredient to prove Theorem \ref{thm_deffn}.

\noindent {\em Proof of Theorem \ref{thm_deffn}.} 
For the proof of part (1), let $M$ be as given, and $S=\{p_1,...,p_n\}$. As each $p_j$, $j=1,...,n$, is a generic CR-singularity of $M$, we can use the biholomorphic equivalence of $(M, p_j)$ and $(\bc_k,O)$, together with Proposition~\ref{prop_Slapar}, to conclude that there exist pairwise disjoint open sets $U_j\ni p_j$ and smooth p.s.h. functions $\psi_j:\C^{3k-1}\rightarrow\rl$, $j=1,...,n$, such that 
\renewcommand{\theenumi}{\alph{enumi}}
	\begin{enumerate}
		\item $\{\psi_j=0\}=M\cap \overline{U_j}$,
		\item $\psi_j>0$ on $\C^{3k-1}\setminus (M\cap \overline{U_j})$, and
		\item $\psi_j$ is strictly p.s.h. on
			 $\C^{3k-1}\setminus\{p_j\}$. 
	\end{enumerate}
Let $\wt M:=M\setminus \bigcup_{1\leq j\leq n}U_j$. Then, as $\wt M$ is totally real in $\C^{3k-1}$, $\psi_0(z):=\dist^2(z,M)$ is strictly p.s.h. on some neighbourhood $U_0$ of $\wt M$ in $\C^{3k-1}$. Now, let $U:=\cup_{0\leq j\leq n}U_j$. The neighbourhoods $U_j$'s in the above construction should be chosen small enough so that $\pi:U\rightarrow M$ given by $z\mapsto p$, where $\dist(z,M)=\dist(p,z)$, is well-defined and smooth. Let $\{\chi_j\}_{0\leq j\leq n}$ be a partition of unity subordinate to $\{U_j\cap M\}_{0\leq j\leq n}$. Define
	\bes
		\rho(z):=\sum_{j=0}^n \chi_j(\pi(z))\psi_j(z).
	\ees 
Since $M\cap U_j\subseteq\{\psi=\nabla \psi_j=0\}$, we have that
	\bes
		dd^c\rho(p)=\sum_0^n \chi_j(p)dd^c\psi_j(p),
	\ees
when $p\in M$. Thus, $dd^c\rho(p)$ is strictly positive on any compact subset of $M\setminus\{p_1,...,p_n\}$. Moreover, since $\chi_j\equiv 1$ near $p_j$, $dd^c\rho=dd^c\psi_j$ near $p_j$. Thus, shrinking $U$ if necessary, we have that $\rho$ is p.s.h. on $U$ and strictly p.s.h. on $U\setminus S$.

To prove $(2)$,  we must extend the form $dd^c\rho$ globally to $\C^{3k-1}$ when $M$ is rationally convex. As a consequence of a characterization of rationally convex hulls due to Duval-Sibony (see~\cite[Remark~2.2]{DuSi95}), there is a smooth p.s.h. function $\theta:\C^{3k-1}\rightarrow\rl$ such that $\om=dd^c\theta$ vanishes on $M'$ and is strictly positive outside $M'$. Once again, we let $\chi$ be a nonnegative smooth function on $\C^{3k-1}$ that is compactly supported in $U$ and identically $1$ on some neighbourhood of $M$ in $U$. For a large enough $C$, the well-defined function $\varphi:=C\theta+\chi\Psi$ is strictly p.s.h. on $\C^{3k-1}\setminus S$. Since the gradient of $\psi$ vanishes along $M'$, we also have that $\iota^*dd^c\varphi=\iota^*dd^c\Psi=d(\iota^*d^c\Psi)=0$, where $\iota:M\rightarrow \C^{3k-1}$ is the inclusion map. Thus, $M$ is isotropic with respect to the degenerate K{\"a}hler form $dd^c\varphi$. 

To prove $(3)$, we must extend $\rho$ globally to $\C^{3k-1}$ when $M$ is polynomially convex. For this, let $\sigma$ be a smooth nonnegative p.s.h. function on $\C^{3k-1}$ so that $M=\sigma^{-1}(0)$ and $\sigma$ is strictly p.s.h. on $\C^{3k-1}\setminus M$. Then, for some nonnegative smooth function $\chi:\C^{3k-1}\rightarrow \rl$ that is compactly supported in $U$ and identically $1$ on some neighbourhood of $M$ in $U$, and for a large enough $C$, we relabel $C\sigma+\chi\rho$ as $\rho$ to obtain the desired extension.
\qed
\section{Appendix: Proof of Lemma~\ref{lem_Slaparfn}}\label{sec_append}
The main technical ingredient of Section \ref{sec_deffn} relies on Lemma~\ref{lem_Slaparfn}. It is a mild generalization of Lemma~$4$ in Slapar's work \cite{Sl04}, whose proof has been omitted there due to its close analogy with the proof of Lemma~$3$ therein. For the sake of completeness, we reproduce Slapar's technique to provide a full proof of Lemma~\ref{lem_Slaparfn}. 

\noindent{\em Proof of Lemma \ref{lem_Slaparfn}.} Recall that
	\bes
		S_\alpha=\left\{Z'=(z,w_1,...,w_{2k-2},w)\in\C^{2k}:
			\begin{aligned}
				&\ima(w_1)=\cdots=\ima(w_{2k-2})=0,\\
					&w=\frac{\alpha}{2}|z|^2+\frac{1}{4}(z^2+\zbar^2)
			\end{aligned}
		\right\}.
	\ees
We consider new real (nonholomorphic) coordinates in $\C^{2k}$, given by
	\begin{align}
		&x=\rea z,\: y=\ima z,\notag\\
		&u_j=\rea w_j, v_j=\ima w_j, \quad 1\leq j\leq 2k-2,\label{eq_coordch}\\
		&u=\rea w-\frac{\alpha}{2}|z|^2-\frac{1}{4}(z^2+\zbar^2),\: v=\ima w.\notag
\end{align}
In these coordinates, $S_\alpha=\{v_1=\cdots =v_{2k-2}=u=v=0\}$. Now, denoting $\partl{}{a}$ and $\secpartl{}{a}{b}$ by $\partial_a$ and $\partial_{a,b}$, respectively, we obtain some of the mixed second-order complex derivatives in the coordinates \eqref{eq_coordch} as follows. 
 \begin{align}
	&4 \partial_{z,\zbar}=\Delta_{x,y}
			-2\left((\alpha+1)x\partial_{x}
				+(\alpha-1)y\partial_{y}+\alpha\right)\partial_{u}\notag\\
					&\qquad \qquad +\Big((\alpha+1)^2 x^2+(\alpha-1)^2 y^2 \Big)\partial_{u},\notag\\
	&4\partial_{z,\wbar}=
		 \partial_{x,u}+\partial_{y,v}
			+i\left(\partial_{x,v}-\partial_{y,u}\right)	
				-\Big((\alpha+1)x-i(\alpha-1)y\Big)\partial_{u,u}\notag\\
					&\qquad \qquad-\Big((\alpha-1)y+i(\alpha+1)x\Big)\partial_{u,v},
						\label{eq_IIder}\\
		&4\partial_{w,\wbar}=\Delta_{u,v}\notag.
\end{align}
Consider the following homogenous polynomial in $\rl[x^2,y^2,u]$ of degree $4$.
	\bes
		P_\alpha(x^2,y^2,u)
			=u^4+\left((4\alpha+c)x^2-cy^2\right)u^3+\left(Ax^4+Bx^2y^2+A'y^4\right)u^2.
	\ees
Using \eqref{eq_IIder}, we have that $4\secpartl{P_\alpha}{z}{\zbar}=$	
\begin{align*}
&\left[\left(6(\alpha+1)^2-3(4\alpha+c)(3\alpha+2)+6A+B\right)x^2+\left(6(\alpha-1)^2+(9\alpha-6)c+B+6A'\right)y^2\right]2u^2\\
	 &{}+\left[(6(\alpha+1)^2(4\alpha+c)-4A(5\alpha+4))x^4
		+(4A'(4-5\alpha)-6(\alpha-1)^2c)y^4\right]u\\
		&+\left[(24\alpha((\alpha-1)^2-c)-20\alpha B)x^2y^2\right]u\\
	 &{}+\left((\alpha+1)^2x^2+(\alpha-1)^2y^2\right)\left(Ax^4+Bx^2y^2+A'y^4\right).
\end{align*}
For this expression to be nonnegative, it suffices for the following equalities and inequalities to hold. 
\renewcommand{\theenumi}{\arabic{enumi}}
	\begin{enumerate}
		\item $A=\dfrac{3(\alpha+1)^2(4\alpha+c)}{2(5\alpha+4)}$;
			\vspace{0.5em}
		\item $A'=\dfrac{3(\alpha-1)^2}{2(4-5\alpha)}c$;
			\vspace{0.5em}
		\item $B=\dfrac{6}{5}((\alpha-1)^2-c)$;
			\vspace{0.5em}
		\item $6A+B+6(\alpha+1)^2>3(4\alpha+c)(3\alpha+2)$;
			\vspace{0.5em}
		\item $6A'+B+6(\alpha-1)^2>(6-9\alpha)c$.
\end{enumerate}
Also, we want that $P_\alpha$ is strictly positive for $u\neq 0$ and $(x,y)$ small. We use the following lemma for this.

\begin{lemma}[{\cite[Lemma $2$.]{Sl04}}]\label{lem_posquad}
Let $p(x,y,u)=u^2+b_1(x,y)u+b_0(x,y)$, where $b_0, b_1$ are continuous functions in a neighbourhood of the origin in $\rl^3$, both vanishing at $(0,0)$. Suppose $b_1^2<4b_0$ for small $(x,y)\neq (0,0)$. Then, there exists a small neighbourhood $U$ of the origin in $\rl^3$ such that $p$ is strictly positive on $U\setminus\{u=0\}$. 
\end{lemma}

\noindent The above lemma yields the following constraints on $A$, $A'$ and $c$. 
\begin{enumerate}
\setcounter{enumi}{5}
\item $(4\alpha+c)^2<4A$;
\item $c^2<4A'$.
\end{enumerate}

To find constants $A$, $B$ and $A'$ that are positive and satisfy inequalities $(4)-(7)$, it suffices to find a $c>0$ such that 
	\be\label{ineq_cond_c}
	c<\min\left\{(\alpha-1)^2,
				\frac{16+8\alpha-64\alpha^2-60\alpha^3}{11+30\alpha+20\alpha^2},
				 \frac{4(\alpha-1)^2(4-5\alpha)}{20\alpha^2-30\alpha+11},	
				  \frac{6-4\alpha-14\alpha^2}{4+5\alpha}
						\right\}.
	\ee
The above condition follows from the positivity assumption on $B$ and by writing inequalities $(4)-(7)$ purely in terms of $c$ and $\alpha$ with the means of $(1)-(3)$. As long as $\alpha<0.46$, the right-hand side of \eqref{ineq_cond_c} is positive. Thus, there exists a homogeneous polynomial $P_\alpha$ of degree $4$ in $\rl[x^2,y^2,u]$ such that 
	\begin{itemize}		
		\item $P_\alpha>0$ for $u\neq 0$ and $(x,y)$ small enough;
\vspace{0.5em} 
		\item $\secpartl{P_\alpha}{z}{\zbar}=0$ when $(x,y)=(0,0)$, but is strictly positive otherwise; and	
\vspace{0.5em} 
		\item $\secpartl{P_\alpha}{z}{\zbar}=q_1u^2+q_3$, where $q_1, q_3\in\rl[x^2,y^2]$ are polynomials of degree $1$ and $3$, respectively, with strictly positive coefficients. 
	\end{itemize}

Now consider 
	\bes
		Q_\alpha(z,w_1,...,w_{2k-2},w)
			=P_\alpha(x^2,y^2,u)+(x^2+y^2)u^4
				+\frac{1}{2}\left(\sum_{j=1}^{2k-2}v_j^2+v^2\right),
	\ees
where the coordinates $(z,w_1,...,w_{2k-2},w)$ and $(x,y,u_1,v_1,...,u_{2k-2},v_{2k-2},u,v)$ are as in \eqref{eq_coordch}. Note that
	\bes
		\secpartl{Q_\alpha}{z}{\zbar}
			=u^4-(40\alpha)u^3+(q_2+\eps q_1)u^2+(1-\eps)q_1u^2+q_3,
	\ees
where $q_2\in\rl[x^2,y^2]$ is of degree $2$. By Lemma~\ref{lem_posquad}, for any $\eps>0$, $u^4-(40\alpha)u^3+(q_2+\eps q_1)u^2$ is strictly positive for $u\neq 0$ and $(x,y)$ small enough. So, there is a neighbourhood $\V_\alpha$ of the origin such that 
		\be\label{eq_zzbar}
			\secpartl{Q_\alpha}{z}{\zbar}\geq R_3, 
		\ee
where $R_3=\sum r_{j,k,l}\ (x^2)^j (y^2)^k u^l$ is a homogeneous polynomial in $\rl[x^2,y^2,u]$ of degree $3$, which is nondegenerate in the sense that all $r_{j,k,l}>0$ whenever $l$ is even. Next, we have that 
	\be\label{eq_wwbar}
		\secpartl{Q_\alpha}{w}{\wbar}
			=\frac{\Delta_{u,v}P_\alpha}{4}+3(x^2+y^2)u^2+\frac{1}{4}>\frac{1}{8}.
	\ee
Using \eqref{eq_IIder}, we also note that
	\begin{align}
		\left|\secpartl{Q_\alpha}{z}{\wbar}\right|^2
		<&\ R_5,\label{eq_zwbar} 
	\end{align}
where $R_5(x^2,y^2,u)$ is some homogeneous polynomial of degree $5$. Combining these estimates, we have that 
	\bes	
		\secpartl{Q_\alpha}{z}{\zbar}\secpartl{Q_\alpha}{w}{\wbar}-
				\left|\secpartl{Q_\alpha}{z}{\wbar}\right|^2
			\geq R_3-R_5,
	\ees
which --- owing to the nondegeneracy of $R_3$ --- is positive on $\V_\alpha$ (shrinking if necessary) as long as $(x,y,u)\neq(0,0,0)$. As the characteristic polynomial of $\Chess Q_\alpha$ (in the variable $\lambda$) is 
	\bes
		\left(\lambda-\frac{1}{4}\right)^{2k-2}
		\left(\lambda^2-\left(\secpartl{Q_\alpha}{z}{\zbar}+\secpartl{Q_\alpha}{w}{\wbar}\right)\lambda
			+\secpartl{Q_\alpha}{z}{\zbar}\secpartl{Q_\alpha}{w}{\wbar}-
				\left|\secpartl{Q_\alpha}{z}{\wbar}\right|^2
	\right),
	\ees	
we obtain that $Q_\alpha$ is a p.s.h. function on $\V_\alpha$ satisfying
	\begin{itemize}
		\item $Q_\alpha^{-1}(0)\cap\V_\alpha=S_\alpha\cap \V_\alpha$.
		\item $Q_\alpha>0$ on $\V_\alpha\setminus S_\alpha$.
		\item $Q_\alpha$ is stirctly p.s.h. on $\V_\alpha\setminus\{x=y=u=0\}$.
	\end{itemize}

To complete the construction of $\rho_\alpha$, let $
		\eta(z,w_1,...,w_{2k-2},w)
			=\left(\frac{1}{2}+x^2+y^2\right)\left(\sum_{j=1}^{2k-2}v_j^2+v^2\right)$.
In a small enough neighbourhood $\V$ of the origin, $\eta$ is p.s.h., and strictly p.s.h. when $(v_1,\cdots,v_{2k-2},v)\neq (0,\cdots,0,0)$. Finally, to obtain the desired neighbourhood and function, set $V_\alpha:=~\V_\alpha\cap~\V$ and $\rho_\alpha:=Q_\alpha+\eta$. This completes the proof of Lemma~\ref{lem_Slaparfn}.
\qed

\bibliography{PolyCvxBiblio}

\begin{thebibliography}{10}

\bibitem{Al98}
H.~Alexander.
\newblock Hulls of subsets of the torus in $\mathbb{C}^2$.
\newblock {\em Annales de l'institut Fourier}, 48(3):785--795, 1998.

\bibitem{ArWo17}
L.~Arosio and E.~F. Wold.
\newblock Totally real embeddings with prescribed polynomial hulls.
\newblock {\em arXiv preprint, https://arxiv.org/pdf/1702.01002.pdf}, 2017.

\bibitem{AuLaPo94}
M.~Audin, F.~Lalonde, and L.~Polterovich.
\newblock Symplectic rigidity: Lagrangian submanifolds.
\newblock In {\em Holomorphic curves in symplectic geometry}, pages 271--321.
  Springer, 1994.

\bibitem{Be97}
V.~K. Beloshapka.
\newblock The normal form of germs of four-dimensional real submanifolds in
  $\mathbb{C}^5$ at generic {$\mathbb{R}\mathbb{C}$}-singular points.
\newblock {\em Math. Notes}, 61(5):777--779, 1997.

\bibitem{Bi65}
E.~Bishop.
\newblock Differentiable manifolds in complex {E}uclidean space.
\newblock {\em Duke Math. J.}, 32:1--21, 1965.

\bibitem{Co97}
A.~Coffman.
\newblock Enumeration and normal forms of singularities in {C}auchy-{R}iemann
  structures.
\newblock {\em {D}issertation, University of Chicago}, 1997.

\bibitem{Co06}
A.~Coffman.
\newblock Analytic stability of the {CR} cross-cap.
\newblock {\em Pacific J. Math.}, 226(2):221--258, 2006.

\bibitem{Do95}
A.~V. Domrin.
\newblock A description of characteristic classes of real submanifolds in
  complex manifolds via {RC}-singularities.
\newblock {\em Izv. Math.}, 59(5):899--918, 1995.

\bibitem{DuSi95}
J.~Duval and N.~Sibony.
\newblock Polynomial convexity, rational convexity, and currents.
\newblock {\em Duke Math. J.}, 79(2):487--513, 1995.

\bibitem{Fo94}
F.~Forstneri{\v c}.
\newblock Approximation by automorphisms on smooth submanifolds of
  $\mathbb{C}^n$.
\newblock {\em Math. Ann.}, 300(1):719--738, 1994.

\bibitem{Fo11}
F.~Forstneri{\v c}.
\newblock {\em Stein manifolds and holomorphic mappings}.
\newblock Springer, 2011.

\bibitem{FoRo93}
F.~Forstneri{\v{c}} and J.-P. Rosay.
\newblock Approximation of biholomorphic mappings by automorphisms of
  $\mathbb{C}^n$.
\newblock {\em Invent. Math.}, 112(1):323--349, 1993.

\bibitem{FoSt91}
F.~Forstneri{\v{c}} and E.~L. Stout.
\newblock A new class of polynomially convex sets.
\newblock {\em Ark. Mat.}, 29(1):51--62, 1991.

\bibitem{GoGu73}
M.~Golubitsky and V.~Guillemin.
\newblock {\em Stable mappings and their singularities}.
\newblock Springer-Verlag, 1973.

\bibitem{GuSh17}
P.~Gupta and R.~Shafikov.
\newblock Rational and polynomial density on compact real manifolds.
\newblock {\em Internat. J. Math.}, 28(05):1750040, 17pp, 2017.

\bibitem{HoJaLa12}
P.~T. Ho, H.~Jacobowitz, and P.~Landweber.
\newblock Optimality for totally real immersions and independent mappings of
  manifolds into $\mathbb{C}^n$.
\newblock {\em New York J. Math.}, 18:463--477, 2012.

\bibitem{Iz18}
A.~J. Izzo.
\newblock No topological condition implies equality of polynomial and rational
  hulls.
\newblock {\em arXiv preprint, https://arxiv.org/pdf/1806.06160.pdf}, 2018.

\bibitem{IzSt18}
A.~J. Izzo and E.~L. Stout.
\newblock Hulls of surfaces.
\newblock {\em Indiana Univ. Math. J.}, to appear.

\bibitem{Jo97}
B.~J{\"o}ricke.
\newblock Local polynomial hulls of discs near isolated parabolic points.
\newblock {\em Indiana Univ. Math. J.}, pages 789--826, 1997.

\bibitem{La72}
H.~F. Lai.
\newblock Characteristic classes of real manifolds immersed in complex
  manifolds.
\newblock {\em Trans. Amer. Math. Soc.}, 172:1--33, 1972.

\bibitem{LoWo09}
E.~L{\o}w and E.~F. Wold.
\newblock Polynomial convexity and totally real manifolds.
\newblock {\em Complex Var. and Elliptic Equ.}, 54(3-4):265--281, 2009.

\bibitem{MI76}
S.~Minsker.
\newblock Some applications of the {S}tone-{W}eierstrass theorem to planar
  rational approximation.
\newblock {\em Proc. Amer. Math. Soc.}, 58(1):94--96, 1976.

\bibitem{NiWe67}
R.~Nirenberg and R.~O. Wells~Jr.
\newblock Holomorphic approximation on real submanifolds of a complex manifold.
\newblock {\em Bull. Amer. Math. Soc.}, 73(3):378--381, 1967.

\bibitem{OFPrWa84}
A.~G. O'Farrell, J.~K. Preskenis, and D.~Walsh.
\newblock Holomorphic approximation in {L}ipschitz norms.
\newblock {\em Proceedings of the conference on Banach algebras and several
  complex variables (New Haven, Conn., 1983) Contemp. Math.}, 32:187--194,
  1984.

\bibitem{RaSi74}
R.~M. Range and Y.-T. Siu.
\newblock $\cont^k$-approximation by holomorphic functions and closed forms on
  $\cont^k$-submanifolds of a complex manifold.
\newblock {\em Math. Ann.}, 210(2):105--122, 1974.

\bibitem{ShSu15}
R.~Shafikov and A.~Sukhov.
\newblock Rational approximation and lagrangian inclusions.
\newblock {\em Enseign. Math.}, 62(3-4):487--499, 2016.

\bibitem{Sl04}
M.~Slapar.
\newblock On {S}tein neighborhood basis of real surfaces.
\newblock {\em Math. Z.}, 247(4):863--879, 2004.

\bibitem{St63}
G.~Stolzenberg.
\newblock Polynomially and rationally convex sets.
\newblock {\em Acta Math.}, 109(1):259--289, 1963.

\bibitem{St07}
E.~L. Stout.
\newblock {\em Polynomial convexity}, volume 261.
\newblock Springer Science \& Business Media, 2007.

\bibitem{VoZa71}
D.~G. Vodovoz and M.~G. Zaidenberg.
\newblock On the number of generators in the algebra of continuous functions.
\newblock {\em Mathematical Notes}, 10(5):746--748, 1971.

\bibitem{We85}
S.~Webster.
\newblock The {E}uler and {P}ontrjagin numbers of an $n$-manifold in
  $\mathbb{C}^n$.
\newblock {\em Comment. Math. Helvetici}, 60:193--216, 1985.

\bibitem{Wi95}
J.~Wiegerinck.
\newblock Local polynomially convex hulls at degenerated {CR} singularities of
  surfaces $\mathbb{C}^2$.
\newblock {\em Indiana Univ. Math. J.}, pages 897--915, 1995.

\end{thebibliography}
\bibliographystyle{plain}
%
\end{document}